\documentclass[12pt]{amsart}
\usepackage{amsmath,amssymb,mathrsfs,euscript,color,stmaryrd,mathabx}
\input{xypic}

\usepackage{enumitem}  

\usepackage{epigraph}
\usepackage{hyperref}
\hypersetup{
pdfpagemode=UseOutlines,      
pdfstartview=Fit,             
pdffitwindow=true,            
pdfpagelayout=TwoColumnsRight,
pdftoolbar=true,              
pdfmenubar=true,              
bookmarksopen=false,          
bookmarksnumbered=true,       
colorlinks=true,              
pdfauthor={Ton nom},          
pdftitle={Titre PDF},         
pdfcreator=PDFLaTeX,          %
pdfproducer=PDFLaTeX,         %
linkcolor=blue,               
urlcolor=blue,                
anchorcolor=black,            
citecolor=blue,               
frenchlinks=true,             
pdfborder={0 0 0}             
}
\usepackage{xcolor}

\newtheorem{Theorem}[equation]{Theorem}

\newtheorem{Criterion}[equation]{Criterion}

\theoremstyle{definition}
\newtheorem{Remark}[equation]{Remark}

\makeatletter
\def\section{\def\@secnumfont{\mdseries}%
  \@startsection{section}{1}%
  \z@{.7\linespacing\@plus\linespacing}{.5\linespacing}%
  {\normalfont\scshape\centering}}
\def\subsection{\def\@secnumfont{\normalfont\bfseries}%
 \@startsection{subsection}{2}%
 \z@{.5\linespacing\@plus.7\linespacing}{-.5em}%
 {\normalfont\bfseries}}
\makeatother

\numberwithin{equation}{section}

\def\bdots{\mathinner{\mkern1mu\raise1pt\hbox{.}\mkern2mu\raise4pt\hbox{.}
           \mkern2mu\raise7pt\vbox{\kern7pt\hbox{.}}\mkern1mu}}

\def\Fx{F^\times}

\def\oF{{\mathfrak o}_F}
\def\oFx{{\mathfrak o}_F^\times}
\def\pF{{\mathfrak p}_F}

\def\oEx{{\mathfrak o}_E^\times}

\def\cInd{\operatorname{c-Ind}}

\def\chara{\operatorname{char}}

\def\Sp{\operatorname{Sp}}

\def\GL{\operatorname{GL}}

\def\tr{\operatorname{tr}}

\begin{document}

\title[Reducibility points and characteristic $p$ local fields I]{Reducibility points and characteristic $p$ local fields I
Simple supercuspidal representations of symplectic groups.}

\author{Corinne Blondel}
\address{CNRS -- IMJ--PRG, Universit{\'e} Paris Cit\'e, Case 7012, 75205 Paris Cedex 13, France.}
\email{corinne.blondel@imj-prg.fr}

\author{Guy Henniart}
\address{Universit{\'e} de Paris-Sud, Laboratoire de Math{\'e}matiques d'Orsay, Orsay Cedex, F-91405.}
\email{guy.henniart@universite-paris-saclay.fr}

\author{Shaun Stevens}
\address{School of Mathematics, University of East Anglia, Norwich Research Park, Norwich NR4 7TJ, United Kingdom}
\email{Shaun.Stevens@uea.ac.uk}


\begin{abstract} Let $F$ be a non-Archimedean local field with {\bf odd} characteristic $p$. Let $N$ be a positive integer and $G=\Sp_{2N}(F)$. By work of Lomelí  on $\gamma$-factors of pairs and converse theorems, a {\bf generic} supercuspidal representation $\pi$ of $G$ has a transfer to a smooth irreducible representation  $\Pi_\pi$ of $\GL_{2N+1}(F)$. In turn the Weil--Deligne representation   $\Sigma_\pi$ associated to $\Pi_\pi$  by the Langlands correspondence determines a Langlands parameter   $\phi_\pi$ for $\pi$. That process produces a Langlands correspondence for generic cuspidal representations of $G$.

In this paper we take $\pi$ to be {\bf simple} in the sense of Gross and Reeder, and from the explicit construction of $\pi$ we describe  $\Pi_\pi$ explicitly. The method we use is the same as in a previous paper, where we treated the case where $F$ is a $p$-adic field, and $\pi$  a simple supercuspidal representation of $G=\Sp_{2N}(F)$. It relies on a criterion due to M\oe glin on the reducibility of representations parabolically induced from $\GL_M(F)\times  G$ for varying positive integers $M$.  We extend this criterion to the case when $F$ has 
{\bf any}  positive characteristic.   The main new feature consists in relating reducibility to $\gamma$-factors for pairs. 
\end{abstract}

\date{\today}
\maketitle

\section{Introduction}

The present paper is an addition to \cite{BHS-Oaxaca}. We establish here for local fields of odd characteristic $p$ what was done there for $p$-adic fields. Let us be more precise.

Let ${p}$ be 
{a prime number}, and ${F}$ a locally compact non-Archimedean field with {\bf residue characteristic  $p$}.
Let  ${N}$ be a positive integer and $G$ the locally pro-$p$ group 
$\Sp_{2N}(F)$. All our representations of $G$  and other reductive groups will be smooth representations on complex vector spaces, and a supercuspidal representation is for us irreducible.
Let $\pi$ be a supercuspidal representation of $G$.

When $\chara (F)=0$,  work of Arthur gives a transfer of $\pi$ to an irreducible representation $\Pi_\pi$ of $\GL_{2N+1}(F)$, characterized via some endoscopic and twisted endoscopic character identities. For generic $\pi$ the transfer can also be obtained by the method of Cogdell, Kim, Piatetski-Shapiro and Shahidi \cite{CKPSS}, using converse theorems and Langlands--Shahidi factors for pairs (see Appendix B in \cite{AHKO}). When 
$\chara (F)=p$,  Lomelí \cite{Lom} developed the Langlands--Shahidi method for pairs and also established the transfer for generic $\pi$.

 \begin{Theorem}
[{\cite[Chapter 7]{CKPSS}},{\cite[Proposition 9.3 and 9.4]{Lom}}]\label{Pisubpi} 

Let $\pi$ be a generic supercuspidal representation of $G=\Sp_{2N}(F)$. Then there is an irreducible generic representation  $\Pi_\pi$ of  $\GL_{2N+1}(F)$ such that, for any non-trivial character $\psi$ of $F$, any positive integer $M$ and any supercuspidal representation $\tau$ of $\GL_M(F)$, we have
$$\gamma(\pi \times \tau, s, \psi)=\gamma\left(\Pi_\pi \times \tau, s, \psi\right).$$
 The representation $\Pi_\pi$ is unique up to isomorphism.
\end{Theorem}

The (class of) $\Pi_\pi$ is called the {\bf transfer} of $\pi$ to $\GL_{2N+1}(F)$.

It is also known  {\cite[Theorem 7.3]{CKPSS}},{\cite[Theorem 9.6]{Lom}}
  that $\Pi_\pi$ is parabolically induced from a tensor product of unitary supercuspidal representations $\Pi_i$ of  $\GL_{N_i}(F)$, where the sum of the $N_i$'s is $2 N+1$, and the $\Pi_i$ 's are self-dual, of orthogonal type (i.e. the Langlands--Shahidi factor 
$L( \Pi_i, \text{Sym}^2, s) $  has a pole at $s=0$), and non-isomorphic to one another. In particular $\Pi_\pi$ is self-dual. It also follows from {\it loc.cit.} that
$ \Pi_\pi^\vee=\Pi_\pi$ is the transfer of $\pi^\vee$.

Let us now assume that~{{\bf $p$ is odd} and}
 $\pi$ is {\bf simple supercuspidal} in the sense of Gross \& Reeder \cite{GR}. A simple supercuspidal representation of $G$ is generic  (see \cite{Oi})  
  so $\pi$ has a transfer $\Pi_\pi$. There is an easy construction of $\pi$ by compact induction (hence the word simple). In \cite{BHS-Oaxaca} we determined $\Pi_\pi$ explicitly from $\pi$ when $\chara (F)=0$. In fact we proved that $\Pi_\pi$ is parabolically induced from the tensor product of an explicit quadratic character $\Pi_{1}$ of $\GL_{1}({~F})={F}^\times$ and an explicit simple supercuspidal representation $\Pi_{2}$ of $\GL_{2N}(F)$. We give a completely similar description here when $\chara (F)=p$.

When $\chara (F)=0$  we used in  \cite{BHS-Oaxaca}  a criterion of M\oe glin to determine the $\Pi_i$'s. That criterion says that an irreducible unitary supercuspidal representation $\rho$ of some 
$\GL_M(F)$ is (isomorphic to) one of the $\Pi_i$'s if and only if the representation of $\Sp_{2M+2N}(F)$   parabolically induced from $\rho |\operatorname{det} |^s \otimes \pi$ reduces at some real number ${s}\ge  1$ (in fact at ${s}=1$  since $\pi$ is generic).

Our main contribution here is to show that this criterion is still valid when $\chara (F)=p$ {(whether or not~$p$ is odd)}.  Indeed in   \cite[section 4]{BHS-Oaxaca}     we determined the set of $\rho$ satisfying the reducibility property above {when~$p$ is odd}, irrespective of the characteristic of $F$, so the reducibility  criterion gives the same description of $\Pi_\pi$ when $\chara (F)=p$.  {We prove the criterion in section~\ref{sec:crit}, and a precise statement (and proof) of our Theorem are in section~\ref{sec:thm}.}

\begin{Remark}  Our result when $\chara (F)=0$ was not new, it had been obtained before by  Adrian \& Kaplan  \cite{AK} and  Oi \cite{Oi} by different methods. We believe that when $\chara (F)>0$ the result is new. In section~\ref{sec:add} we also comment on other approaches, when $\chara (F)>0$.
\end{Remark}

\bigskip 
{\bf Acknowledgements.}      The second-named author would like to thank  Wee-Teck Gan, Luis Lomelí and Hirotaka Kakuhama for communications. He would also like to thank the Department of Mathematics of UMPC Valparaíso and Luis Lomelí for the invitation which initiated the present paper. The third-named author was supported by EPSRC grant EP/V061739/1.

\section{M\oe glin's criterion and $L$-functions for pairs}\label{sec:crit}
In this section, {\bf $p$ is any prime number,}
 $\pi$ is a {\bf generic supercuspidal} representation of 
$\Sp_{2N}(F)$  and, as in the introduction, the $\Pi_i$, for ${i}=1, \ldots , {r}$, are the supercuspidal representations of 
$\GL_{M_i}(F)$   appearing in the supercuspidal support of $\Pi_\pi$. M\oe glin's criterion to determine the $\Pi_i$ is the following.

\begin{Criterion}\label{criterion}
 Let $\rho$ be a unitary supercuspidal representation of some $\GL_M(F)$. Then the following conditions are equivalent:

\begin{enumerate}
\item $\rho$ is isomorphic to one of the $\Pi_i$.

\item  The representation $I(\rho \times \pi, s)$ of $\Sp_{2 M+2 N}(F)$ parabolically induced from $\rho |\operatorname{det} |^s \otimes \pi$ is reducible at some real number ${s}_{0}\ge  1$.
\end{enumerate} 
\end{Criterion}

When $\operatorname{char}({F})=0$, see \cite{Moeglin}. We prove below that criterion when $\operatorname{char}({F})>0$.

\begin{Remark}  We shall show that in (ii) reducibility occurs at ${s}_{0}=1$.
\end{Remark}

\begin{Remark}  To be specific about condition (ii) we embed the product 
$\GL_M(F) \times \Sp_{2N}(F)$  as a Levi subgroup $L$ of $H=\Sp_{2M+2N}(F) $  by sending $(g,h)$ to the block-diagonal matrix with blocks\footnote{With the symplectic form used in \cite{BHS-Oaxaca} $g^T$ is the transpose of $g$ with respect to the antidiagonal; in general $g \mapsto g^T$ is the unique anti-involution of ${\mathbb M}_M (F)$ such that the block-diagonal matrix belongs to 
$H$.}  ${g} , {h}$ and $({g}^{T})^{ -1}$, and we choose the upper triangular parabolic subgroup ${P}$ of ${H}$ with Levi subgroup $L$.
The parabolic induction we use is the normalized one.
\end{Remark}

 \begin{proof}[Proof of Criterion~\ref{criterion}]
Reducibility of ${I}(\rho \times \pi, s)$ is governed by Harish-Chandra's $\mu$-function $\mu(s, \rho\otimes\pi, \psi)$, as recalled in Lemma 7.6 of \cite{GL}, which gathers results of Harish-Chandra and Silberger valid for $F$ of any characteristic.
We deduce that ${I}(\rho \times \pi, s)$ is irreducible for any real ${s}$ if $\rho$ is not self-dual, and that if $\rho$ is self-dual, there is a unique ${s}_{0}\ge 0$ where ${I}(\rho \times \pi, s_0)$ reduces. Also ${s}_{0}=0$ if and only if $\mu(0, \rho\otimes\pi, \psi)$ is non-zero, and if $\mu(0, \rho\otimes\pi, \psi)=0$, then $s_{0}>0$ is the only pole of $\mu(s, \rho\otimes\pi, \psi)$ for $s\ge 0$. If $\rho$ is not self-dual, conditions (i) and (ii) are not satisfied, so we assume $\rho$ self-dual.

Since we assume $\pi$ generic, $\mu(s, \rho\otimes\pi, \psi)$ can be expressed as 
$$\gamma(\rho \times \pi^\vee, s, \psi) \gamma\left(\rho^\vee \times \pi,-s, \bar\psi \right) \gamma\left(\rho, 2 s, \Lambda^{2}, \psi\right) \gamma\left(\rho^\vee,-2 s, \Lambda^{2}, \bar\psi \right)$$  
(up to a non-zero constant), where the last two factors are Langlands--Shahidi factors for exterior square. To determine $s_{0}$ in terms of those $\gamma$-factors, there are two ways: either we follow the pattern of Theorem 8.1 of \cite{Shahidi90}, which, when applied to the classical group $H$ with its Levi subgroup $L$ implies the criterion when $\operatorname{char}({F})=0$, or as in \cite{GL} we take advantage of the work of Laurent or Vincent Lafforgue to write all $\gamma$-factors as factors attached to Weil--Deligne representations, and work on the Weil group side.
We choose to follow Shahidi here, because Lomelí  has  already verified most of the ingredients Shahidi uses. More specifically,  by   {\cite[bottom of page 4314]{Lom}},  
Propositions~7.2 and~7.4, Lemma~7.5 and Corollary~7.6 of \cite{Shahidi90} are valid when $\operatorname{char}({F})>0$. In terms of the notation in {\it loc.cit.} $\gamma(\rho \times \pi^\vee, s, \psi)$ is Shahidi's $\gamma_{1}(\rho\otimes\pi, s, \psi)$, and $\gamma\left(\rho, 2 s, \Lambda^{2}, \psi\right)$ is Shahidi's $\gamma_{2}(\rho\otimes\pi, 2s, \psi)$, which does not depend on $\pi$.

Write $\gamma(\rho \times\pi^\vee, s, \psi)=\varepsilon(\rho \times \pi^\vee, s, \psi) L(\rho^\vee \times \pi  , 1-s) / L(\rho \times\pi^\vee, s)$ where the $L$-factors are the Langlands--Shahidi factors (which do not depend on $\psi)$. 
 They are obtained by writing   $\gamma(\rho \times \pi^\vee, s, \psi)$  as 
$e\left(q^{ -s}\right) P\left(q^{-s}\right) / Q\left(q^{-s}\right)$, where $P$ and $Q$ are  two coprime polynomials in one variable which take the value $1$ at $0$, and $e$ is a monomial.  Then  one puts   ${L}(\rho \times \pi^\vee, {s})=1 / {P}\left({q}^{-{s}}\right)$ and $\varepsilon(\rho \times \pi^\vee, s, \psi)={e}\left({q}^{-{s}}\right)$.

Similarly write 
$\gamma\left(\rho, 2 s, \Lambda^{2}, \psi\right)=\varepsilon\left(\rho, 2 s, \Lambda^{2}, \psi\right) {L}\left(\rho^\vee, 1-{s}, \Lambda^{2}\right) / {L}\left(\rho, \Lambda^{2}, {s}\right)$ by expressing $\gamma\left(\rho, 2 {s}, \Lambda^{2},\psi\right)=e_{2}\left(q^{-s}\right) P_{2}\left(q^{-s}\right) / Q_{2}\left(q^{-s}\right)$ where $P_{2}$ and $Q_{2}$ are two coprime polynomials in one variable which take the value $1$ at $0$, and $e_{2}$ is a monomial. Proposition 7.3 of \cite{Shahidi90}  says that all roots of $P$ and ${P}_{2}$ have absolute value $1$ (and consequently all roots of ${Q}$ and ${Q}_{2}$ have absolute value ${q}$), and Corollary 7.6 of {\it loc.cit.}  implies that the following conditions are equivalent:
\begin{enumerate}[label=\alph*)] 
\item $\mu(0, \rho\otimes\pi, \psi)=0$ (that is, $s_{0}>0$ ).

\item We have either ${P}(0)=0$ or ${P}_{2}(0)=0$ but not both.
\end{enumerate} 

Let us examine in turn the two conditions ${P}(0)=0$ and ${P}_{2}(0)=0$.

Since $\Pi_\pi^\vee$ is the transfer of $\pi^\vee$, we have 
$$\gamma(\rho \times \pi^\vee, s, \psi)=\gamma\left(\rho \times  \Pi_\pi^\vee, s, \psi\right)$$ and that is the product over $i$ of the $\gamma\left(\rho \times \left(\Pi_i\right)^\vee , s,  \psi  \right)$, and similarly for the $\varepsilon$ and $L$-factors. Moreover it is known that all those factors for $\rho \times\left(\Pi_{i}\right)^\vee$ can also be obtained by the Rankin--Selberg method
\cite{HL2013}. In particular 
$L(  \rho \times  (  \Pi_i)^\vee, s)$  has a pole at $s=0$ if and only if $\rho$ is isomorphic to $\Pi_i$, so $\gamma(\rho \times \pi^\vee, 0, \psi)$
is $0$, equivalently ${P}(0)=0$, if and only if $\rho$ is isomorphic to one of the $\Pi_i$'s. 

On the other hand ${P}_{2}(0)=0$ means that ${L}\left(\rho, \Lambda^{2}, {s}\right)$ has a pole at ${s}=0$, i.e. $\rho$ is of symplectic type. Then ${P}(0)$ is not 
$0$, so that $\rho$ is isomorphic to none of the $\Pi_i$   (which have orthogonal type), which in turn implies that the factor $\gamma(\rho \times \pi^\vee, s, \psi)$ has no zero nor pole for $s \ge 0$, and the same is true of the factor 
$\gamma\left(\rho^\vee \times \pi,-s, \bar\psi\right)$. On the other  hand   the only pole of   $\gamma\left(\rho, 2 s, \Lambda^{2}, \psi\right)$ for $s \ge 0$ is at $s=1 / 2$, and the same is true of $\gamma\left(\rho^\vee,-2 s, \Lambda^{2}, \bar\psi\right)$.

We deduce that if $s_{0}>0$, then either $\rho$ is isomorphic to one of the $\Pi_i$, in which case $s_{0}=1$, or $\rho$ is isomorphic to none of the $\Pi_i$, in which case $s_{0}=1 / 2$.

That proves the implication (ii) implies (i) in the criterion, and the fact that $s_{0}=1$ if (ii) is satisfied. On the other hand if (i) is satisfied, then clearly $\rho$ is self-dual of orthogonal type (because the $\Pi_i$  are) and in particular ${P}_{2}(0)$ is not zero since $\rho$ cannot be of symplectic type, and the above analysis implies that (ii) holds with ${s}_{0}=1$.
\end{proof}

\section{The Theorem}\label{sec:thm}
From now on (in this section and the next), we assume that~{\bf $p$ is odd}. 

To state the result properly we need more notation. We denote by $\oF$ the ring of integers of $F$, by $\pF$  its maximal ideal, and by $q$ the cardinality of the residual field 
$\oF/\pF$. We fix a character $\Psi$ of $F$ non-trivial on  $\oF$ and trivial on 
$\pF$. For  a given $\beta \in M_{2N}(F)$ we denote by $\Psi_\beta$ the function on $  \GL_{2N}(F)$ defined by 
$\Psi_\beta(X)=\Psi \circ \tr (\beta (X-1))$. We use it on restriction to relevant subgroups 
of $  \GL_{2N}(F)$, depending on $\beta$, on which it is a character.

 We let   $\delta$ be  the (non-trivial) quadratic character of $\oFx$. 
 The Gauss sum $\sum_{u \in k_F^\times} \delta(u)  \Psi( u  )$ is  known to have  
modulus $q^{\frac 1 2}$ and to satisfy 
$\xi(\delta, \Psi)^2=(-1)^{\frac{q-1}{2}}$  where we have put
\begin{equation}\label{Gauss}
\xi(\delta, \psi)= \frac{\sum_{u \in k_F^\times} \delta(u)  \Psi( u  )}{|\sum_{u \in k_F^\times} \delta(u)  \Psi( u  )|}  ,   \text{ a fourth root of unity}. 
\end{equation}

We described in \cite[\S 2 and \S 4.7]{BHS-Oaxaca} the 
simple supercuspidal representations of $G=\Sp_{2N}(F)$, according to the definition in \cite{GR}, 
in our notation (a  description already given in   
\cite{AK} and \cite{Oi}). We recall this briefly.  

We fix an Iwahori subgroup $\tilde I$ of $\GL_{2N}(F)$ 
which is stable under the involution whose fixed points are  $G$,   and the two first steps of its Moy-Prasad filtration $\tilde I(1)$ and $\tilde I(2)$. Intersecting with $G$ we obtain  an Iwahori subgroup $I$ of $G$ and the two first steps $I(1)$ and $I(2)$  of its Moy--Prasad filtration. 

We let $E$ be a totally   ramified extension  of $F$ of degree $2N$ normalizing  $I (1)$ and $\beta^{-1}$ be a uniformizer of $E$. Then $\Psi_\beta$ is a character of 
$I(1)$ trivial on $I(2)$ and for any character $\chi$ of the center $Z$ of $G$, isomorphic to $\{\pm 1\}$,    the compactly induced representation $\cInd_{Z I_{2N}(1)}^G \chi \otimes \Psi_\beta$ is irreducible hence supercuspidal. By definition \cite{GR}, such a representation is called a simple supercuspidal representation.

When $F$ has characteristic $0$, we proved   in  \cite[Theorem 4.6]{BHS-Oaxaca} 
the Theorem that follows. Here we establish the same statement
when $\chara (F)=p$.

\begin{Theorem}\label{JordanSummary} Let $\pi$ be a simple supercuspidal representation of $G$, written as 
 $$\pi=\cInd_{Z I (1)}^G \chi \otimes \Psi_\beta,$$ where 
 $\chi$ is a character of  the center $Z \simeq\{\pm 1\}$ of $G$ and $\beta^{-1}$ 
is  a uniformizer of a totally ramified extension $E$ of $F$ of degree $2N$ normalizing  $I (1)$.\\
The supercuspidal support  of $\Pi_\pi$ is $(\GL_1(F)\times \GL_{2N}(F), 
\Pi_1 \otimes \Pi_2)$
where  
\begin{itemize} 
\item   $\Pi_1$ is the ramified quadratic  character of $\Fx$ characterized by 
$$\Pi_1(   N_{E/F}(\beta))=   (-1)^{(N+1)\frac{q-1}{2}} ; $$
\item    $\Pi_2$ is  the  
 supercuspidal representation of $\GL(2N,F)$  defined by 
 $$
\Pi_2= \cInd_{E^\times \tilde I (1)}^{\GL(2N,F)}  \tau_\beta\otimes \Psi_{2\beta}  
$$
where $(\tau_\beta)_{|\oEx} $ is the quadratic character of $\oEx$ and 
$$\tau_\beta(\beta) =\chi(-1)\delta(2)  \xi(\delta, \Psi ).$$
\end{itemize}
\end{Theorem}

\begin{proof}
In   \cite[Theorem 4.16]{BHS-Oaxaca}     we proved that the representations $\Pi_{1}$ and $\Pi_{2}$ 
satisfy condition (ii) of Criterion~\ref{criterion}, with reducibility at ${s}_{0}=1$, hence also condition (i). Since $\Pi_{1}$ is a representation of $\GL_{1}({F})$ and $\Pi_{2}$ a representation of $\GL_{2N}(F)$, while $\Pi_\pi$ is a representation of $\GL_{2N+1}(F)$, we see that $\Pi_{1}$ and $\Pi_{2}$ are the only representations in the supercuspidal support of $\Pi_\pi$. 
\end{proof}

Note that we only used the implication (ii) implies (i) in Criterion~\ref{criterion}.

\section{Other approaches and possible extensions}\label{sec:add}
In  \cite[5.2 and 5.3]{BHS-Oaxaca}   we gave another proof of our main theorem when $\operatorname{char}({F})=0$, which uses \cite{BHS} but not the more involved computations of \cite[\S 3 \& \S 4]{BHS-Oaxaca}. But we needed the supplementary information given by a result of Lapid, saying that for a generic 
 supercuspidal  representation $\pi$ of $G$ with central character $\chi$, we have $\varepsilon(\pi, 1 / 2, \psi)=\chi(-1)$. Lapid's paper \cite{La} is written with the hypothesis that $\operatorname{char}(F)=0$, and similarly for the more general result of Lapid \& Rallis on non-generic representations, though they say that the hypothesis is only for convenience (Lapid--Rallis \cite{LR} when $\operatorname{char}(F)=0$, see recent work of Kakuhama \cite{Kakuhama}  when $\operatorname{char}(F)=p$).   Taking that result for granted\footnote{H. Kakuhama tells us that the main point would be to prove Lemma 13 of Lapid--Rallis when $\operatorname{char}(F)=p$.},  
 the proof of \cite[5.2 and 5.3]{BHS-Oaxaca} goes through when $\operatorname{char}(F)=p$. On the other hand, by \cite[4.8]{BHS-Oaxaca},  our Theorem implies in turn the identity
$\varepsilon(\pi, 1 / 2, \psi)=\chi(-1)$, when $\pi$ is simple supercuspidal.

As mentioned in the introduction, when $\operatorname{char}({F})=0$, Adrian and Kaplan, and also Oi, had proved a result equivalent to our main Theorem, though in different notation.

Adrian and Kaplan \cite{AK} used \cite{BHS}, which determined $\Pi_{1}$ and $\Pi_{2}$ up to twist by an unramified character, together with some explicit computations of $\gamma$-factors $\gamma(\omega \times \pi, s, \psi)$ for quadratic characters $\omega$ of $F^{\times}={GL}_{1}({~F})$. Their paper is written with the hypothesis that $\operatorname{char}({F})=0$, but we presume that their computations can be adapted when $\operatorname{char}({F})>0$.

Oi \cite{Oi}   worked directly with the character identities of endoscopy and twisted endoscopy characterizing Arthur's transfer. The present state of the (twisted) trace formula does not allow to derive Arthur's results when $\operatorname{char}({F})>0$ (see recent progress on the trace formula in \cite{LL} and \cite{Le}).

But Ganapathy and Varma (\cite{Ga}, \cite{GV}) used Kazhdan's theory of close local fields to transport Arthur's transfer, for split classical groups, from $\operatorname{char}({F})=0$ to  $\operatorname{char}({F})>0$. To deduce our Theorem one would need to show that both the construction of simple 
 supercuspidals  and the answer given by our Theorem 
are compatible with that transport. That should not be hard, but we do not embark on that, because our main point was to show that the arguments of \cite{BHS},  \cite{BHS-Oaxaca} are enough to give the result when $\operatorname{char}({F})>0$. 
Note however that Lapid’s result mentioned above can be easily transfered from $\operatorname{char}({F})=0$ to $\operatorname{char}({F})>0$ by the work of \cite{Ga}. One can probably also prove in this way the extension (by Lapid and Rallis \cite{LR}) of Lapid’s result to not necessarily generic representations, using that Kakuhama \cite{Kakuhama}  has established the doubling method for  $\operatorname{char}({F})>0$, but the behaviour of factors obtained via the doubling method through close local fields has not been studied yet, and probably it is easier to work directly in Kakuhama’s framework.

Simple  supercuspidals  exist for other (tame) classical groups, and when $\operatorname{char}({F})=0$ their transfer has been described explicitly (see the references in \cite{AHKO}, and also \cite{Tam}).
When $\operatorname{char}({F})>0$ that remains to be done, although the use of close local fields, for example, seems promising.

On the other hand, when $\operatorname{char}({F})>0$, Gan and Lomelí \cite{GL}, using the work of Vincent Lafforgue, construct a transfer $\Pi_\pi$ to 
$\GL_{2N+1}(F)$   for any  supercuspidal  representation $\pi$ of 
$\Sp_{2N}(F)$, not necessarily generic. However they work under an assumption (working hypothesis), which they do not verify: when $\operatorname{char}({F})=0$ the working hypothesis is true by work of Savin \cite{Sa2008}. 
In future work we plan to prove that hypothesis when  $\operatorname{char}({F})=p$, and consequently show that the results of \cite{BHS} allow to compute the supercuspidal support of $\Pi_\pi$ up to unramified character twists, and that similarly the results of \cite{LS}   hold when  $\operatorname{char}({F})=p$.


\begin{thebibliography}{10} 

\bibitem{AHKO}
Moshe Adrian, Guy Henniart, Eyal Kaplan, Masao Oi. 
\newblock Simple supercuspidal $L$-packets of split special orthogonal groups over dyadic fields.
\newblock https://arxiv.org/abs/2305.09076

\bibitem{AK} Moshe Adrian and Eyal Kaplan. 
\newblock The Langlands parameter of a simple supercuspidal representation: symplectic groups.
\newblock  Ramanujan J. 50 (2019), no.3, 589--619.



\bibitem{Arthur}
James Arthur. \newblock
{\em The endoscopic classification of representations}, volume 61 of {\em American Mathematical Society Colloquium Publications}. 
\newblock American Mathematical Society, Providence, RI, 2013. 

\bibitem{BHS}  Corinne Blondel, Guy Henniart, Shaun Stevens.     \newblock  Jordan blocks of cuspidal representations of symplectic groups. \newblock  {\em Algebra Number Theory}    12 (2018), no. 10, 2327--2386.  

\bibitem{BHS-Oaxaca}  Corinne Blondel, Guy Henniart, Shaun Stevens.     \newblock 
Simple cuspidal representations of symplectic groups: Langlands parameter.     \newblock
https://arxiv.org/abs/2310.20455
 

\bibitem{CKPSS} J.W. Cogdell, H.H. Kim, I.I. Piatetski-Shapiro, F. Shahidi. 
\newblock  On lifting from classical groups to $\GL_N$. Publ. Math. Inst. Hautes Études Sci. 93(1) (2001)   5--30.  

\bibitem{GL} 
Wee Teck Gan and Luis Lomel\'\i. \newblock 
Globalization of supercuspidal representations over function fields and applications. 
\newblock  {\em
J. Eur. Math. Soc.} 20 (2018), no.11, 2813--2858

\bibitem{GV}  
Radhika Ganapathy  and  Sandeep Varma. \newblock  
On the local Langlands correspondence for split classical groups over local function fields. \newblock {\em 
J. Inst. Math. Jussieu} 16 (2017), no.5, 987--1074.



\bibitem{Ga}  Radhika Ganapathy. \newblock   A Hecke algebra isomorphism over close local fields. \newblock  {\em Pacific J. Math.} 319 (2022), no. 2, 307--332.

\bibitem{GR}
Benedict~H. Gross and Mark Reeder.
\newblock Arithmetic invariants of discrete {L}anglands parameters.
\newblock {\em Duke Math. J.}, 154(3) (2010), 431--508.

\bibitem{HL2013} 
Guy Henniart and Luis  Lomelí. \newblock
Uniqueness of Rankin-Selberg products. \newblock
J. Number Theory 133 (2013), no.12, 4024--4035.   

\bibitem{Kakuhama}   Hirotaka Kakuhama. \newblock
On the local doubling $\gamma$-factor for classical groups over function fields. 
\newblock 
{\em J. Number Theory}  233 (2022), 337--369.

\bibitem{LL}  
Jean-Pierre Labesse, Bertrand Lemaire.
\newblock La formule des traces tordue pour les corps de fonctions. \newblock 
https://doi.org/10.48550/arXiv.2102.02517



\bibitem{La} Erez M.  Lapid. \newblock On the root number of representations of orthogonal type. \newblock {\em Compos. Math.} 140 (2004), no. 2, 274--287.

\bibitem{LR} Erez 
Lapid and Stephen Rallis. \newblock 
On the local factors of representations of classical groups. \newblock  
Automorphic representations, L-functions and applications: progress and prospects, 309–359.
Ohio State Univ. Math. Res. Inst. Publ., 11
Walter de Gruyter \& Co., Berlin, 2005.   

\bibitem{Le}  Bertrand Lemaire. \newblock  
Développement fin de la contribution unipotente à la formule des traces sur un corps global de caractéristique $p>0$, I.  \newblock   	
https://doi.org/10.48550/arXiv.2212.03792







\bibitem{Lom} Luis Lomel\'\i. \newblock
Functoriality for the classical groups over function fields. 
\newblock {\em Int. Math. Res. Notices} (2009), 4271--4335.

\bibitem{LS}
 Jaime Lust and Shaun Stevens. \newblock 
On depth zero $L$-packets for classical groups. \newblock
{\em Proc. Lond. Math. Soc.} (3) 121 (2020), no.5, 1083--1120.

 

\bibitem{Moeglin}
Colette M{\oe}glin.
\newblock Paquets stables des s\'eries discr\`etes accessibles par endoscopie tordue; leur param\`etre de {L}anglands.
\newblock
pp. 295--336 in {\em Automorphic forms and related geometry: assessing the legacy of I. I. Piatetski-Shapiro (New Haven, CT, 2012)}.
\newblock Contemp. Math. 614, Amer. Math. Soc., Providence, RI, 2014.

\bibitem{Oi} Masao Oi.    
\newblock Simple supercuspidal $L$-packets of quasi-split classical groups. 
\newblock  	
https://arxiv.org/abs/1805.01400,  {\it to appear in Mem. AMS}.


\bibitem{Sa2008}  Gordan Savin. \newblock  
Lifting of generic depth zero representations of classical groups.
\newblock  {\em 
J. Algebra} 319(2008), no.8, 3244--3258.   
  
\bibitem{Shahidi90} Freydoon  
Shahidi. \newblock  
A proof of Langlands' conjecture on Plancherel measures; complementary series for p-adic groups. \newblock  
Ann. of Math. (2) 132 (1990), no.2, 273--330.



\bibitem{Tam}
Geo Kam-Fai Tam.
\newblock Endoscopic liftings of epipelagic representations for classical groups. 
\newblock https://arxiv.org/abs/2311.02812

\end{thebibliography}
\end{document}